\documentclass[11pt,reqno]{amsart}

\usepackage{amsmath,amsfonts,amsthm,amssymb}
\usepackage{graphicx}
\usepackage{verbatim}
\usepackage{color}
\usepackage{amscd}
\usepackage{verbatim}
\usepackage{dsfont}

\begin{document}
	\newcommand{\Green}[4]{\mbox{$G^{#1}_{#2}(#3,#4)$}}
\newcommand{\M}{{\mathcal M}}
\newcommand{\loc}{{\mathrm{loc}}}
\newcommand{\core}{C_0^{\infty}(\Omega)}
\newcommand{\sob}{W^{1,p}(\Omega)}
\newcommand{\sobloc}{W^{1,p}_{\mathrm{loc}}(\Omega)}
\newcommand{\merhav}{{\mathcal D}^{1,p}}
\newcommand{\be}{\begin{equation}}
\newcommand{\ee}{\end{equation}}
\newcommand{\mysection}[1]{\section{#1}\setcounter{equation}{0}}
\newcommand{\laplace}{\Delta}
\newcommand{\pl}{\laplace_p}
\newcommand{\grad}{\nabla}
\newcommand{\pd}{\partial}
\newcommand{\bo}{\pd}
\newcommand{\csub}{\subset \subset}
\newcommand{\sm}{\setminus}
\newcommand{\ssm}{:}
\newcommand{\diver}{\mathrm{div}\,}
\newcommand{\bea}{\begin{eqnarray}}
\newcommand{\eea}{\end{eqnarray}}
\newcommand{\bean}{\begin{eqnarray*}}
\newcommand{\eean}{\end{eqnarray*}}
\newcommand{\thkl}{\rule[-.5mm]{.3mm}{3mm}}
\newcommand{\cw}{\stackrel{\rightharpoonup}{\rightharpoonup}}
\newcommand{\id}{\operatorname{id}}
\newcommand{\supp}{\operatorname{supp}}
\newcommand{\wlim}{\mbox{ w-lim }}
\newcommand{\mymu}{{x_N^{-p_*}}}
\newcommand{\R}{{\mathbb R}}
\newcommand{\N}{{\mathbb N}}
\newcommand{\Z}{{\mathbb Z}}
\newcommand{\Q}{{\mathbb Q}}
\newcommand{\abs}[1]{\lvert#1\rvert}
\newtheorem{theorem}{Theorem}[section]
\newtheorem{corollary}[theorem]{Corollary}
\newtheorem{lemma}[theorem]{Lemma}
\newtheorem{notation}[theorem]{Notation}
\newtheorem{definition}[theorem]{Definition}
\newtheorem{remark}[theorem]{Remark}
\newtheorem{proposition}[theorem]{Proposition}
\newtheorem{assertion}[theorem]{Assertion}
\newtheorem{problem}[theorem]{Problem}
\newtheorem{conjecture}[theorem]{Conjecture}
\newtheorem{question}[theorem]{Question}
\newtheorem{example}[theorem]{Example}
\newtheorem{Thm}[theorem]{Theorem}
\newtheorem{Lem}[theorem]{Lemma}
\newtheorem{Pro}[theorem]{Proposition}
\newtheorem{Def}[theorem]{Definition}
\newtheorem{defi}[theorem]{Definition}
\newtheorem{Exa}[theorem]{Example}
\newtheorem{Exs}[theorem]{Examples}
\newtheorem{Rems}[theorem]{Remarks}
\newtheorem{Rem}[theorem]{Remark}

\newtheorem{Cor}[theorem]{Corollary}
\newtheorem{Conj}[theorem]{Conjecture}
\newtheorem{Prob}[theorem]{Problem}
\newtheorem{Ques}[theorem]{Question}
\newtheorem*{corollary*}{Corollary}
\newtheorem*{theorem*}{Theorem}
\newtheorem{thm}[theorem]{Theorem}
\newtheorem{lem}[theorem]{Lemma}
\newtheorem{prop}[theorem]{Proposition}
\newtheorem{cor}[theorem]{Corollary}
\newtheorem{ex}[theorem]{Example}
\newtheorem{rem}[theorem]{Remark}
\newtheorem*{thmm}{Theorem}
\newcommand{\Hmm}[1]{\leavevmode{\marginpar{\tiny%
			$\hbox to 0mm{\hspace*{-0.5mm}$\leftarrow$\hss}%
			\vcenter{\vrule depth 0.1mm height 0.1mm width \the\marginparwidth}%
			\hbox to
			0mm{\hss$\rightarrow$\hspace*{-0.5mm}}$\\\relax\raggedright #1}}}

\newcommand{\pf}{\noindent \mbox{{\bf Proof}: }}


\renewcommand{\theequation}{\thesection.\arabic{equation}}
\catcode`@=11 \@addtoreset{equation}{section} \catcode`@=12
\newcommand{\Real}{\mathbb{R}}
\newcommand{\real}{\mathbb{R}}
\newcommand{\Nat}{\mathbb{N}}
\newcommand{\ZZ}{\mathbb{Z}}
\newcommand{\CC}{\mathbb{C}}
\newcommand{\Pess}{\opname{Pess}}
\newcommand{\Proof}{\mbox{\noindent {\bf Proof} \hspace{2mm}}}
\newcommand{\mbinom}[2]{\left (\!\!{\renewcommand{\arraystretch}{0.5}
\mbox{$\begin{array}[c]{c}  #1\\ #2  \end{array}$}}\!\! \right )}
\newcommand{\brang}[1]{\langle #1 \rangle}
\newcommand{\vstrut}[1]{\rule{0mm}{#1mm}}
\newcommand{\rec}[1]{\frac{1}{#1}}
\newcommand{\set}[1]{\{#1\}}
\newcommand{\dist}[2]{$\mbox{\rm dist}\,(#1,#2)$}
\newcommand{\opname}[1]{\mbox{\rm #1}\,}
\newcommand{\mb}[1]{\;\mbox{ #1 }\;}
\newcommand{\undersym}[2]
 {{\renewcommand{\arraystretch}{0.5}  \mbox{$\begin{array}[t]{c}
 #1\\ #2  \end{array}$}}}
\newlength{\wex}  \newlength{\hex}
\newcommand{\understack}[3]{%
 \settowidth{\wex}{\mbox{$#3$}} \settoheight{\hex}{\mbox{$#1$}}
 \hspace{\wex}  \raisebox{-1.2\hex}{\makebox[-\wex][c]{$#2$}}
 \makebox[\wex][c]{$#1$}   }%
\newcommand{\smit}[1]{\mbox{\small \it #1}}
\newcommand{\lgit}[1]{\mbox{\large \it #1}}
\newcommand{\scts}[1]{\scriptstyle #1}
\newcommand{\scss}[1]{\scriptscriptstyle #1}
\newcommand{\txts}[1]{\textstyle #1}
\newcommand{\dsps}[1]{\displaystyle #1}
\newcommand{\dx}{\,\mathrm{d}x}
\newcommand{\dy}{\,\mathrm{d}y}
\newcommand{\dz}{\,\mathrm{d}z}
\newcommand{\dt}{\,\mathrm{d}t}
\newcommand{\dr}{\,\mathrm{d}r}
\newcommand{\du}{\,\mathrm{d}u}
\newcommand{\dv}{\,\mathrm{d}v}
\newcommand{\dV}{\,\mathrm{d}V}
\newcommand{\dW}{\,\mathrm{d}W}
\newcommand{\ds}{\,\mathrm{d}s}
\newcommand{\dS}{\,\mathrm{d}S}
\newcommand{\dk}{\,\mathrm{d}k}
\newcommand{\dm}{\,\mathrm{d}m}
\newcommand{\dmu}{\,\mathrm{d}\gm}

\newcommand{\dphi}{\,\mathrm{d}\phi}
\newcommand{\dtau}{\,\mathrm{d}\tau}
\newcommand{\dxi}{\,\mathrm{d}\xi}
\newcommand{\deta}{\,\mathrm{d}\eta}
\newcommand{\dsigma}{\,\mathrm{d}\sigma}
\newcommand{\dtheta}{\,\mathrm{d}\theta}
\newcommand{\dnu}{\,\mathrm{d}\nu}

\def\ga{\alpha}     \def\gb{\beta}       \def\gg{\gamma}
\def\gc{\chi}       \def\gd{\delta}      \def\ge{\varepsilon}
\def\gth{\theta}                         \def\vge{\varepsilon}
\def\gf{\phi}       \def\vgf{\varphi}    \def\gh{\eta}
\def\gi{\iota}      \def\gk{\kappa}      \def\gl{\lambda}
\def\gm{\mu}        \def\gn{\nu}         \def\gp{\pi}
\def\vgp{\varpi}    \def\gr{\rho}        \def\vgr{\varrho}
\def\gs{\sigma}     \def\vgs{\varsigma}  \def\gt{\tau}
\def\gu{\upsilon}   \def\gv{\vartheta}   \def\gw{\omega}
\def\gx{\xi}        \def\gy{\psi}        \def\gz{\zeta}
\def\Gg{\Gamma}     \def\Gd{\Delta}      \def\Gf{\Phi}
\def\Gth{\Theta}
\def\Gl{\Lambda}    \def\Gs{\Sigma}      \def\Gp{\Pi}
\def\Gw{\Omega}     \def\Gx{\Xi}         \def\Gy{\Psi}

\renewcommand{\div}{\mathrm{div}}
\newcommand{\red}[1]{{\color{red} #1}}

%


\newcommand{\De} {\Delta}
\newcommand{\la} {\lambda}
\newcommand{\bn}{\mathbb{B}^{2}}
\newcommand{\rn}{\mathbb{R}^{2}}
\newcommand{\bnn}{\mathbb{B}^{N}}
\newcommand{\rnn}{\mathbb{R}^{N}}
\newcommand{\kP}{k_{P}^{M}}

\newcommand{\authorfootnotes}{\renewcommand\thefootnote{\@fnsymbol\c@footnote}}%


\def\e{{\text{e}}}
\def\N{{I\!\!N}}

\numberwithin{equation}{section} \allowdisplaybreaks

\title[Perturbation theory of positive solutions]{Some new aspects of perturbation theory of positive solutions of second-order linear elliptic equations}
\author{Debdip Ganguly}
\address{Debdip Ganguly, Department of Mathematics,  Indian Institute of Science Education and Research, Dr. Homi Bhabha Road, Pune 411008, India}
\email{debdipmath@gmail.com}
\author{Yehuda Pinchover}
\address{Yehuda Pinchover,
Department of Mathematics, Technion - Israel Institute of
Technology,   Haifa 3200003, Israel}
\email{pincho@technion.ac.il}


\date{}

\begin{abstract}

We present some new results concerning perturbation theory for positive solutions of second-order linear elliptic operators,  including further study of the equivalence of positive minimal Green functions and the validity of a Liouville comparison principle for nonsymmetric operators.


\vspace{.2cm}

\noindent  2000  \! {\em Mathematics  Subject  Classification.}
{Primary 35B09; Secondary 31C35, 35A08, 35J08.} \\[1mm]

\end{abstract}
\maketitle

 \section{Introduction}\label{sec_int}
 
 Let $M$ be a smooth, connected,  and noncompact  Riemannian manifold of dimension $N$. We consider a
  second-order elliptic operator $P$ with real coefficients in the divergence form
\begin{equation} \label{operator}
Pu:=-\div\! \left[A(x)\nabla u +  u\tilde{b}(x) \right]  +
 b(x)\cdot\nabla u   +c(x)u \qquad x\in M.
\end{equation}
More precisely, let $m>0$ be a strictly positive measurable function in $M$ such that $m$ and $m^{-1}$  are bounded on any compact subset of $M$, and denote $\dm: =m(x)\!\dx$, where $\dx$ is the Riemannian volume form of $M$ (which is just the Lebesgue measure in the case of Schr\"odinger operators on domains of $\R^N$). 

\medskip

We denote by $T_xM$ and $TM$ the tangent space to $M$ at $x\in M$ and the tangent bundle, respectively. Let $\mathrm{End}(T_xM)$ and $\mathrm{End}(TM)$ be the set of endomorphisms in $T_xM$ and the corresponding
bundle, respectively. The gradient with respect to the Riemannian metric
is denoted by $\nabla$, and $-\div$ is the formal adjoint of the gradient with respect to the measure ${\rm d}m$. The inner product and the induced norm on $TM$ are
denoted by $\langle X, Y\rangle$ and $|X|$, respectively, where $X, Y \in TM$.

\medskip

We assume that $A$ is a symmetric measurable section on $M$ of $\mathrm{End}(TM)$ such that for any compact set $K$ in $M$ there exists a positive constant $\lambda_K\geq 1$ satisfying
\begin{equation}\label{ell}
\lambda_K^{-1} |\xi|^2
\leq |\xi|^2_{A(x)}  :=\langle A(x)\xi, \xi\rangle 
\leq \lambda_K  |\xi|^2  
\qquad \forall x\in K \mbox{ and } (x,\xi)\in TM.
\end{equation}
We assume also that the coefficients $b$ and $\tilde b$ are measurable vector fields in
$M$ of class $L^p_\loc(M)$  and $c$ is a measurable function in $M$ of class $L^{p/2}_\loc(M)$ for some $p > N$.

We denote by $P^\star$ the formal adjoint operator of $P$ on its natural space $L^2(M,\!\dm)$. 
When $P$ is in divergence form (\ref{operator}) and $b = \tilde{b}$, then the operator
\be\label{symm_P}
Pu = - \div \left[ \big(A \grad u + u b\big) \right] + b \cdot \grad u + c u,
\ee
is {\em symmetric} in the space $L^2(M, \!\dm)$. Throughout the paper, we call this setting the {\em symmetric case}. 
 We note that if $P$ is symmetric and $b$ is smooth enough, then $P$ is in fact a Schr\"odinger-type operator of the form
\be\label{eq-symm}
Pu =  - \div \big(A \grad u \big) + \tilde{c} u,
\ee
where $\tilde{c}=c-\div\, b$.
\medskip 

By a {\em solution} $v$ of the equation $Pu = 0$,  we mean $v \in W^{1,2}_{\loc}(M)$ that satisfies the equation in the {\em weak sense}. Subsolutions and supersolutions are defined similarly. 

 Denote the cone of all positive solutions of the
equation $Pu=0$ in $M$ by $\mathcal{C}_{P}(M)$. Let $V$ be a real valued potential. The {\em
generalized principal eigenvalue} of the operator $P$ and a potential $V\in L^q_\loc(M)$, $q>N/2$,   is defined by
$$\gl_0(P,V,M)
:= \sup\{\gl \in \mathbb{R} \; \mid\; \mathcal{C}_{P-\lambda V}(M)\neq
\emptyset\}.$$
We say that $P$ is {\em nonnegative in} $M$ (and we denote it by $P\geq 0$ in $M$) if $\lambda_0:= \lambda_0(P,\mathbf{1},M)\geq 0$,
 where $\mathbf{1}$ is the constant function on $M$ taking
at any point $x\in M$ the value $1$. Throughout the paper we always assume that
$\gl_0\geq 0$, that is, $P\geq 0$ in $M$.

\medskip

The main purpose of the paper is to present some new results concerning perturbation theory of the cone $\mathcal{C}_{P}(M)$. Perturbation theory of positive solutions was studied extensively in the past few decades. S.~Agmon in \cite{AG1, AG2} studied positivity and decay properties of solutions of second-order elliptic equations using the notion of \emph{Agmon ground state}. His results turned out to be highly influential in the study of the structure of $\mathcal{C}_{P}(M)$ and its behaviour under certain types of perturbations (the so-called {\em criticality theory}). 
Without any claim of completeness, we refer to some relevant papers studying criticality theory \cite{AN, AA, GH, MHM,MM0,MM1, YP3, YP1, YP2,YP5, Pinsky95} and references therein. 

\medskip

The perturbation that we consider here is of the form $P_\gl:= P-\gl V$, where $P\geq 0$ in $M$, $\gl\in \R$ and $V\in L^q_\loc(M)$, $q>N/2$. 
We study, in particular, the maximal interval such that the Green function of $P_\gl$ is equivalent to  the Green function of $P$, certain classes of `big' and `small' perturbations,  compactness properties of weighted Green operators for certain classes of `small' weights, and a new Liouville comparison principle for nonsymmetric operators. See Section~\ref{sec-AO} for more details.

\medskip

The outline of our paper is as follows. In Section~\ref{sec-pre} we recall some definitions and basic known results concerning criticality theory, and in Section~\ref{sec-AO} we discuss the problems that we study in the present paper. Section~\ref{section_maximal_green} is devoted to our results concerning the equivalence of positive minimal Green functions of second-order elliptic operators under 
nonnegative perturbation. 
In Section~\ref{sec_hbig} we prove that \emph{optimal} Hardy-weights are \emph{h-big} perturbations in the sense of \cite{GH}, 
while in Section~\ref{sec-critical-Hardy} we present a large family of 
\textquoteleft small\textquoteright\, Hardy-weights $W_\gm$, given by a simple explicit formula, such that $P-W_\gm$ is positive-critical. In Section~\ref{sec-torsion} we prove that for symmetric operators, the assumption of finite torsional rigidity implies that the spectrum of $P$ on $L^2(M,\!\dm)$ is discrete.  Section~\ref{Sec_4.1} is devoted to a Liouville comparison principle for 
{\em nonsymmetric}, nonnegative, elliptic operators. We conclude our paper in Section~\ref{sec_green_function_hyperbolic} where we apply perturbation theory to study the asymptotic of the positive minimal Green function of the shifted 
Laplace-Beltrami operator on the hyperbolic space $\mathbb{H}^N$.


\section{Preliminaries}\label{sec-pre}
In the present section we fix our setting and notation, and recall some basic definitions and results concerning criticality theory. 

 Let $M$ be a smooth, connected,  and noncompact  Riemannian manifold of dimension $N$, and $P$ an elliptic operator of the form \eqref{operator}. Throughout the paper we use the following notation.
\begin{itemize}
	\item We denote  by $ \infty $ the ideal point which is added to $ M $ to obtain the one-point compactification of $M$. 
	
	\item  We write $X_1 \Subset X_2$ if the set $X_2$ is open in $M$, the set $\overline{X_1}$ is
	compact and $\overline{X_1} \subset X_2$. 
	
	\item Let $g_1,g_2$ be two positive functions defined in a domain $D$. We say that $g_1$ is {\em equivalent} to $g_2$ in $D$ (and use the notation $g_1\asymp g_2$ in
	$D$) if there exists a positive constant $C$ such
	that
	$$C^{-1}g_{2}(x)\leq g_{1}(x) \leq Cg_{2}(x) \qquad \mbox{ for all } x\in D.$$
	
	\item We fix a {\em compact exhaustion} of $M$, i.e., a sequence of smooth relatively
	compact domains in $M$ such that $M_1 \neq \emptyset,$ $M_j \Subset M_{j + 1}$ and 
	$\cup_{j = 1}^{\infty} M_j = M$. We denote  $M_j^* := M \setminus \overline{M_j}.$
	\item  We denote the restriction of a function $f:M\to \R$ to $A\subset M$ by $f \!\!\upharpoonright_A$.
\end{itemize}
\medskip

We first recall the definitions of critical and subcritical operators and of a ground
 state (for more details on criticality theory, see \cite{MM0,MM1,YP3, YP1, YP2} and references therein).
\begin{defi}\label{groundstate}{\em
Let $K \Subset M$. We say that $u \in \mathcal{C}_{P}(M \setminus K)$ is a {\em positive solution of the
 operator $P$ of minimal growth in a neighborhood of infinity in $M$}, if for
any compact set $K \Subset K_{1} \Subset M$ with a smooth boundary and any positive supersolution $v$ of the equation $Pw=0$
 in $M \setminus K_{1}$,   $v\in C((M \setminus K_{1})\cup \partial K_{1})$, the inequality
$u \leq v$ on $\partial K_{1}$ implies that $u \leq v$ in $M \setminus K_{1}$.

A positive solution $u \in \mathcal{C}_{P}(M)$ which has minimal growth in a neighborhood of infinity in
$M$ is called the \emph{(Agmon) ground state} of $P$ in $M$ (see \cite{AG2}).
}
 \end{defi}
\begin{defi}\label{critical}{\em
The operator $P$ is said to be {\em critical} in $M$ if $P$ admits a ground state in $M$. The operator $P$ is called
 {\em subcritical} in $M$ if  $P\geq 0$ in $M$ but $P$ is
not critical in $M$. If $P \not\geq 0$ in $M$, then $P$ is said to be {\em supercritical} in $M$. 
 }
\end{defi}
If $W\in L^q_\loc(M;\R_+)$ with $q>N/2$ is a nonzero nonnegative potential, then $P- \lambda W$ is subcritical for every $\lambda \in (-\infty, \lambda_0(P, W, M))$, and supercritical for $\lambda > \lambda_0(P, W,M)$. Furthermore, if $P$ is critical in $M$, then $\lambda_0(P, W, M) = 0$. 
\begin{rem}\label{altenatecritical}{\em
Let $P\geq 0$ in $M$. It is well known that the operator $P$ is critical in $M$ if and only if the equation $P u = 0$ in $M$ has 
a unique (up to a multiplicative constant) 
positive supersolution (see \cite{MM1,YP3}). In particular, if $P$ is critical in $M$, then $\dim  \mathcal{C}_{P}(M) = 1$. Further, in the critical case, the unique positive supersolution (up to a multiplicative positive constant) is 
a ground state of $P$ in $M$. 

On the other hand, $P$ is subcritical in $M$ if and only if $P$ admits a (unique) positive minimal Green function 
$G_{P}^{M}(x,y)$ in $M$. Moreover, for any fixed $y\in M$, the function $G_{P}^{M}(\cdot,y)$ is a positive solution of minimal growth
 in a neighborhood of infinity in
$M$.  Since,  $G_{P^\star}^{M}(x,y)=G_{P}^{M}(y,x)$, it follows that $P$ is critical (resp. subcritical) in $M$ if and only if $P^\star$ is critical (resp. subcritical) in $M$.
}
\end{rem}
\begin{rem}\label{altenatecritical1}{\em
In the critical case there exists a (sign-changing) \emph{Green function} which is bounded above by the corresponding ground state away from the singularity, see \cite{DP}. 
}
\end{rem}
\begin{defi}\label{null-critical}{\em
1. We say that $W\gneqq 0$ is a {\em Hardy-weight}  of $P$ in $M$ if $P-W\geq 0$ in $M$.  

\medskip 

2. Assume that $W\gneqq 0$ is a Hardy-weight of $P$ in $M$, and that $P-W$ is critical in $M$. Let $\phi$ and $\phi^\star$ be the ground states of $P-W$ and $P^\star-W$, respectively. The operator $P-W$ is said to be {\em null-critical} (respect., {\em positive-critical}) in $M$ with respect to $W$ if $\phi\phi^\star \not\in L^1(M,W\!\dx)$ (respect., $\phi\phi^\star  \in L^1(M,W\!\dx)$). 
 }
\end{defi}
Fix a potential $V\in L^q_\loc(M;\R)$, where  $q>N/2$. Set $S: = S_+ \cup S_0$, where 
\begin{align*}
S_+: &= S_+ (P, V, M) = \{ t \in \mathbb{R} : P - tV \ \mbox{is subcritical in $M$} \},\\[2mm]
S_0: &= S_0 (P, V, M) = \{ t \in \mathbb{R} : P - tV \ \mbox{is critical in $M$} \}.
\end{align*}
Then $S$ is a closed interval and $S_0 \subset \partial S$ \cite{YP2}. Moreover, if $V$ has compact support in $M$, then  $S_0 = \partial S$. In particular, 
subcriticality is stable under compact perturbation, i.e., if $P$ is subcritical and $V$ is a nonzero potential with compact support in $M$,
then there exists $\varepsilon > 0$ such that $P - \varepsilon V$ is subcritical for $|\varepsilon| < \varepsilon_0$ (see \cite{YP1, YP2}).

The above stability property of subcritical operators and other positivity properties are preserved under a larger (and in fact maximal) class of potentials $V$ called \emph{small perturbations} \cite{YP1}. 
We recall below the definition of small perturbation and other types of perturbations by a potential $V$ and discuss briefly some of their properties. 

\begin{defi}[\cite{MM1,YP1}]
{\rm
Let $P$ be a subcritical operator in $M$ and let $V \in L^{q}_\loc(M)$ for some $q > N/2$ be a real valued potential. We say that $V$
is a \emph{small (semismall) perturbation} of $P$ in $M$ if 

\begin{equation*}\label{defi_small_perturbation}
 \lim_{n \rightarrow \infty} \left\{ \sup_{x, y \in M_n^*} \int_{M_n^*} \dfrac{G^M_P(x, z) |V(z)| G^M_P(z, y)\dm(z)}{G^M_P(x, y)}   \right\} = 0,
\end{equation*}

\medskip

\begin{equation*}\label{defi_semismall_perturbation}
\left(\!\! \lim_{n \rightarrow \infty} \!\!\left\{\! \sup_{ y \in M_n^*}\! \int_{M_n^*} \!\!\!\!\dfrac{G^M_P(x_0, z) |V(z)| G^M_P(z, y)\!\dm(z) }{G^M_P(x_0, y)} \!\!\right\} \!\!=\! 0,\! \mbox{ where } x_0 \in M  \mbox{ is fixed}\!\!\right)\!\!.
\end{equation*}
}
\end{defi}

\begin{defi}\label{g_bounded_defi}
{\rm
We say that $V$ is a {\em $G$-(semi)bounded perturbation} of  $P$ in $M$ if there exists a positive constant $C_0$ such that 

\begin{equation}\label{supremum}
C_0 : = \sup_{x, y \in M}  \int_{M}\frac{ G_{P}^{M}(x, z) |V(z)| G_{P}^{M} (z, y) \dm(z)}{G_{P}^{M}(x, y)} < \infty,
\end{equation}

\medskip

\begin{equation*}\label{defi_bdd_perturbation}
\left( \sup_{ y \in M} \int_{M}\dfrac{G^M_P(x_0, z) |V(z)| G^M_P(z, y)\!\dm(z) }{G^M_P(x_0, y)} <\infty, \mbox{ where } x_0 \in M  \mbox{ is fixed}\!\right)\!\!.
\end{equation*}
}
\end{defi}

\begin{rem}
{\rm
A small perturbation is semismall and $G$-bounded \cite{MM1}. On the other hand, if $V$ is $G$-bounded perturbation of $P$ in $M$, and $f$ is an arbitrary bounded function vanishing at infinity in $\Gw$ (i.e. with respect of the one-point compactification of $M$), then clearly, $fV$ is a small perturbation of $P$ in $M$. 
}
\end{rem}

\begin{defi}{\em 
Let $P_i,$ $i = 1, 2$ be two subcritical operators in $M.$ We say that the Green functions $G^{M}_{P_1}(x, y)$ 
and $G^{M}_{P_2}(x, y)$ are {\em equivalent} (respect., {\em semiequivalent}) if 
$G^{M}_{P_1} \asymp G^{M}_{P_2}$ on $M \times M \setminus \{ (x, x) : x \in M \}$ (respect., if for a fixed  $y\in M$, we have $G^{M}_{P_1}(\cdot,y) \asymp G^{M}_{P_2}(\cdot,y)$ on $M \setminus \{ y\}$).
}
\end{defi}
 In the sequel we use the notation 
\begin{align*}
&E_+ \!= \!E_+(P, V, M) := \{ t \in \mathbb{R} \!\mid\! G^{M}_{P - tV} \asymp G^{M}_{P} 
\quad \mbox{on} \  M \times M \setminus \{ (x, x) : x \in M \} ,\\[2mm]
&SE_+ = SE_+(P, V, M) := \{ t \in \mathbb{R} \!\mid\! G^{M}_{P - tV} \mbox{ is semiequivalent to } G^{M}_{P}  \} .
\end{align*}

\begin{rem}\label{rem_sp}
{\rm
Clearly, $E_+\subseteq S_+$.  It is known that if the operator $P$ is subcritical and $V$ is a small perturbation of $P$ in $M,$ then  $E_+ = S_+$,  $\partial S = S_0$, and the corresponding ground states are equivalent to $G_P^M(x,x_0)$ in $M\setminus B(x_0,\vge)$ for sufficiently small $\vge>0$. 

On the other hand, If $V$ is a $G$-bounded perturbation of $P$ in $M,$ then $G^{M}_{P}\asymp G^{M}_{P- tV}$  on $M \times M \setminus \{ (x, x) : x \in M \}$ provided
$|t|$ is small enough \cite{MM1,YP3,YP1}. 
Furthermore, if $G^{M}_{P}(x, y)$ and $G^{M}_{P- V}(x, y)$ are equivalent and $V$ has a definite sign, then $V$ is a
G-bounded perturbation of $P$
in $M$. Moreover, in this case, $E_+$ is an open half-line which is contained in $S_+\setminus  \{ \lambda_0 \}$ \cite[Corollary~3.6]{YP2}. 
}
\end{rem}

\medskip

Finally, we discuss sufficient conditions for the compactness of the following weighted Green operators  with weight $W\geq 0$. Let 
\begin{equation}\label{weighted Green}
\mathcal{G} \! f(x)\!:= \!\!\! \int_{M}\!\!\!  \Green{M}{P}{x}{y}W(y)f(y)\!\dm(y),\;\; \mathcal{G} ^\odot \!f(y) \!:=\!\!\!
\int_{M}\!\!\!  \Green{M}{P}{x}{y}W(x)f(x)\!\dm(x)
\end{equation}
in certain weighted $L^p$ spaces, where $1\leq p\leq \infty$. Let $\gf$ and $\tilde \gf$ be a pair of two positive continuous functions on $M$, and set
$$L^{p}(\gf_p):=L^{p}(M,(\gf_p)^p\!\dm),\quad  L^{p}(\tilde{\gf}_p):=L^{p}(M,(\tilde{\gf}_p)^p\!\dm),$$
 where 
\begin{equation}\label{eq:2.9}
\gf_p:=\gf^{-1}(\gf
W\tilde{\gf})^{1/p},  \qquad \tilde{\gf}_p:=\tilde{\gf}^{-1}(\gf
W\tilde{\gf})^{1/p}.
\end{equation}
We have
\begin{theorem}[\cite{YP17}]\label{thmcomp}
Let $P$ be a subcritical operator in $M$. Assume that $W>0$ is a semismall perturbation of $P^\star $ and $P$ in $M$, and let $\gl_0:=\gl_0(P,W,M)$. 
 Then 
 \begin{enumerate}
 \item The operator  $P-\gl_0W$ is positive-critical with respect to $W$,  that is, 
\be\label{phiwphi} \int_M
\tilde{\gf}(x)W(x)\gf(x)\dm(x)<\infty,
\end{equation}   
where $\gf$ and $\tilde{\gf}$ denote the ground states of $P-\gl_0 W$
and $P^\star-\gl_0 W$, respectively. Moreover, $\gl_0= \|\mathcal{G}\|_{L^p(\gf_p)}^{-1}> 0$ for any  $1\leq p\leq \infty$.
\item  for any $1\leq
p\leq \infty$, the integral operators
$\mathcal{G}$ and $\mathcal{G} ^\odot$ defined in \eqref{weighted Green}
are compact on $L^{p}(\gf_p)$ and $L^{p}(\tilde{\gf}_p)$,
respectively.

\item For $1\leq p\leq \infty$, the spectrum of $\mathcal{G} \!\!\upharpoonright_{L^{p}(\gf_p)}$ contains $0$, and besides, consists of at most a sequence of eigenvalues of
finite multiplicity which has no point of accumulation except $0$.

\item For any $1\leq p\leq\infty$, $\gf$ (resp.    $\tilde{\gf}$) is the unique nonnegative  eigenfunction of the operator
$\mathcal{G} \!\!\upharpoonright_{L^p(\gf_p)}$ (resp.,
$\mathcal{G} ^\odot\!\!\upharpoonright_{L^p(\tilde{\gf}_p)}$). The corresponding eigenvalue
$\gn=(  \gl_0)^{-1}$ is simple.

\item The spectrum of $\mathcal{G} \!\!\upharpoonright_{L^{p}(\gf_p)}$  is
$p$-independent for all $1\leq p\leq \infty$, and we have
$$0\in \gs\left(\mathcal{G} \!\!\upharpoonright_{L^p(\gf_p)}\right) = \gs\left(\mathcal{G} ^\odot\!\!\upharpoonright_{L^p(\tilde{\gf}_p)}\right)
\subset\overline{ B\Big(0, (  \gl_0)^{-1}\Big)} .$$

\item  Suppose further that $P$ is symmetric. Let  $\phi_k$ be the $k$-th (weighted) eigenfunction in $L^2(M, W\!\dm)$ (counting multiplicity). Then for each $k\geq 1$, the quotient of the eigenfunctions $\phi_k/\phi$ is bounded in $M$ and has a continuous extension up to the Martin boundary of the pair $(M, P)$. 
\end{enumerate}
\end{theorem}
\begin{remark}{\em 
		We would like to point out that criticality theory, and
		in particular the results of this paper, are also valid for the
		class of {\em classical solutions} of locally uniformly elliptic operators of the form
		\begin{equation}
		\label{L}
		Lu:=-\sum_{i,j=1}^N a^{ij}(x)\partial_{i}\partial_{j}u + b(x)\cdot\nabla u+c(x) u,
		\end{equation}
		with real and locally H\"older continuous coefficients,
		and for the class of {\em strong solutions} of locally uniformly elliptic operators of the form \eqref{L} with locally bounded coefficients (provided that the formal adjoint operator  also satisfies the same assumptions), see \cite{YP3, YP1, YP2,YP5, Pinsky95} and references therein. Nevertheless, for the
		sake of clarity, we prefer to present our results only for operators in divergence form \eqref{operator} and weak solutions.
	}
\end{remark}

\section{Aims and objectives}\label{sec-AO}
In this section we present the problems that we study in our paper. 
\subsection{Maximal interval of equivalence}

The following problem was posed in  \cite[Conjecture~3.7]{YP2}, see also \cite[Example~8.6]{YP5} for a counterexample. 
 \begin{problem}\label{pb_equivalence}
 Suppose that $P$ is subcritical in $M$ of the form \eqref{operator}, and assume that $W \geq 0$ is a $G$-bounded perturbation of $P$ in $M$. Is it true that 
 $$
 E_+ = S_+ \setminus \{ \lambda_0 \}   ?
 $$
 \end{problem}
 In Section~\ref{section_maximal_green} we provide a positive answer to the above question if $P$ is \emph{symmetric} 
  and its positive minimal Green function satisfies the \emph{quasimetric property}. See also Lemma~\ref{lem_2}, where we  prove that $SE_+ = S_+ \setminus \{ \lambda_0 \} $ for a certain family of   
  nonnegative $G$-semibounded perturbations of a subcritical operator $P$ in $M$.
\subsection{$h$-big perturbation}

Next, we discuss a class of perturbations known as {\em $h$-big perturbations}. This notion was introduced by 
A.~Grigor'yan and W.~Hansen \cite{GH} for the case when $P = -\Delta$,  and later it was generalized 
by M.~Murata (see \cite{MM2, MM3}) for elliptic operators of the form \eqref{operator}. 
\begin{defi}\label{h_big_defi}
{\rm Suppose that $P$ of the form \eqref{operator} is subcritical in $M$. Let $h$ be a positive supersolution of the equation 
$$
P \, u = 0 \quad \mbox{in} \ M.
$$
We say that a nonnegative potential $W$ is a {\em $h$-big in $M$} if there is no function  satisfying 
$$
(P + W) v = 0 \quad \mbox{in } M \mbox{ and } 0<v \leq h \quad \mbox{in a neighborohood of infinity in } M.  
$$
 Otherwise, $W$ is said to be {\em non-$h$-big}. 
}
\end{defi}

\begin{rem}
{\rm 
It is evident from the definition of $h$-big perturbation that it generalizes the following Liouville property
for Schr\"odinger equation \cite{AG}: 

Let $M$ be a smooth, noncompact Riemannian manifold $M$ and let $W\neq 0$ be a smooth nonnegative potential 
on $M$. We say that the operator $-\Delta + W$ satisfies the {\em Liouville property} if    
\begin{equation}\label{liouville_laplace}
(-\Delta + W)u = 0 \quad \mbox{in } M, \mbox{ and } 0\leq u\in L^\infty(M), 
\end{equation}
implies $u= 0$.
}
\end{rem}

Clearly (see for example \cite{AG}), if $W\gneqq 0$ has a compact support the above Liouville property holds true if and only if 
$P:=-\Delta$ is critical in $M$ (in other word, $M$ is parabolic). On the other hand, if $P=-\Delta$ is subcritical in $M$ and    
$$
\int_{M} G_P^M(x, y) W(y) \dm(y) < \infty,
$$
then the Liouville property does not hold \cite{AG, GH}. Moreover, it follows from  \cite[Proposition~3.4]{YP5} that if $P$ is subcritical operator in $M$ of the form  \eqref{operator}, and $h\in  \mathcal{C}_P(M)$, then  $W\gneqq 0$ is non-$h$-big if 
$$
\int_{M} G_P^M(x, y) W(y) h(y) \dm(y)  < \infty.
$$

\medskip 

For a given subcritical operator $P$ of the form \eqref{operator} there is a natural class of weights satisfying $\gl_0(P,W,M)>0$, which are `big' in a certain sense.  

\begin{defi}[\cite{DFP}]\label{def-opt-w}
	{\rm 
we say that $W\gneqq 0$ is an  {\em optimal-Hardy} weight for $P$ in $M$ if the following three properties hold: 
\begin{itemize}
\item {\bf Criticality:} $P-W$ is critical in $M$, and let $\vgf$ and $\vgf^*$ be the corresponding ground states of $P-W$ and $P^*-W$.

\item {\bf  Optimality at infinity:} for any $\gl > 1$ and $K\Subset M$,  $P -\gl W\not\geq 0$ in $M\setminus K$.

\item  {\bf  Null-criticality:} $\vgf\vgf^*\not\in L^1(M,W\!\dm)$.

\end{itemize}
}
\end{defi}
The following theorem is a version of \cite[Theorem~4.12]{DFP} (cf.  the discussion therein).  
\begin{thm}\label{thm_DFP}
	Let $P$ be a subcritical operator in $M$ and let $G^{M}_P(x,y)$ be its minimal positive Green function.
	Let $u\in  \mathcal{C}_P(M)$ satisfying 
	
\begin{equation}\label{DFP_cond}
\lim_{x \rightarrow \infty} \frac{G_P^M(x,y)}{u(x)} = 0,
\end{equation}
	where $\infty$ is the ideal point in the one-point compactification of $M$.
	
	 Let $\gf\gneqq 0$ be a compactly supported smooth function, and consider its Green potential 
	 $$G_\gf(x):=\int_{M} G_P^M(x, y) \gf(y) \dm(y).$$
	
	   Then 
	   \begin{equation}\label{W-Hardy}
 W:=\frac{P(\sqrt{G_\gf u})}{\sqrt{G_\gf u}}
	   \end{equation}
	    is an optimal Hardy-weight for $P$ in $M$. Moreover, 
	   	$$
	W(x) := \frac{1}{4} \left| \nabla \log\left( \frac{G_\gf(x)}{u(x)} \right) \right|^2_{A(x)} \qquad \mbox{in } M\setminus \supp{\gf}.
	$$ 
\end{thm}
We omit the proof of Theorem~\ref{thm_DFP} since it can be obtained by  a slight modification of the proof of  \cite[Theorem~4.12]{DFP}.

In Section~\ref{sec_hbig}, we discuss the following problem.
  \begin{problem}
Study the $h$-bigness property of optimal Hardy-weights $W$  given by Theorem~\ref{thm_DFP}. 
\end{problem}
\subsection{ Critical Hardy-weights}
An important feature of classical Hardy-weights $W$ is the knowledge of the  best Hardy constant. In other words, for such Hardy-weights the value of $\gl_0(P,W,M)$ is known (in contrary to the case of a general weight). 
 We note that the problem of finding a critical potential for a given subcritical operator was studied in \cite[Section~5]{PT06}. The answer obtained there relies on solving a nontrivial auxiliary  variational problem.  Moreover, this variational approach is obviously restricted to {\em symmetric} subcritical operators. 
 
 In Section~\ref{sec-critical-Hardy} we  prove  for any subcritical operator $P$ of the form \eqref{operator},  the existence of a large family of critical Hardy-weights which are given by a simple explicit formula. More precisely, we present a family of \textquoteleft small\textquoteright\, Hardy-weights $W_\gm$ such that each $W_\gm$ is semismall perturbation of $P$ in $M$, and $P-W_\gm$ is positive critical with respect to $W_\gm$. In particular, $\gl_0(P,W_\gm,M)=1$. Recall that {\em optimal} Hardy-weights $W$ given by Theorem~\ref{thm_DFP} are $h$-big and $P-W$ is null-critical with respect to $W$.   
\subsection{ Liouville comparison principle}
Next, we recall a Liouville comparison principle for nonnegative Schr\"odinger-type operators. 
\begin{thm}\cite[Theorem~1.7]{YP07}\label{YP07-thm}
Let $N \geq 1$ and $M$ be a noncompact connected Riemannian manifold.  
Consider two Schr\"odinger operators defined on $M$
of the form \eqref{eq-symm}, that is,
$$
P_j := -\div (A_j \nabla) + V_j \qquad j = 0,1,
$$
such that $A_j$ satisfy \eqref{ell}, and $V_j \in L^{q}_{\loc} (M)$ for some $q > N/2$,  where $j = 0,1$. 

Suppose that the following assumptions hold true:
\begin{enumerate}

\item The operator $P_1$ is critical in $M$. Denote by $\Phi$ be its ground state.

\item $P_0$ is nonnegative in $M$, and there exists a real function $\Psi \in H^{1}_{\loc}(M)$ 
such that $\Psi_{+} \neq 0$, and $P_0 \Psi \leq 0$ in $M$,
    where $u_+(x) := \max \{ 0, u(x) \}.$

\item  The following inequality holds:
\begin{equation*}
(\Psi_+)^2(x) A_0(x) \leq C \Phi^2(x) A_1(x) \qquad  \mbox{ a.e. in } M,
\end{equation*}
where $C > 0$ is a positive constant, and the matrix inequality $A\leq  B$ means
that $B - A$ is a positive semi-definite matrix. 
\end{enumerate}
Then the operator $P_0$ is critical in $M$ and $\Psi$ is its ground state.
\end{thm}

\medskip

We note that in Theorem~\ref{YP07-thm}  there is no assumption on the difference of the given potentials $V_j$. In \cite[Problem~5]{YP07} the author proposed to  generalize Theorem~\ref{YP07-thm} to the case of {\em nonsymmetric} elliptic operators of the form \eqref{operator} with the same (or even with comparable) principal parts. In a recent paper \cite{ABG},  the authors gave a partial answer to the above problem using  a probabilistic approach along with criticality theory under some assumptions on the difference of the given potentials.  

In Section~\ref{Sec_4.1}, we prove another version of Liouville comparison principle for nonsymmetric nonnegative operators. In particular, we provide a quantitative bound on the difference of the given potentials in terms 
of a certain Hardy-weight to guarantee  the validity of a Liouville comparison principle. Moreover, in contrast to  \cite[Theorem~2.3]{ABG} which holds in $\R^N$,  our result holds in any noncompact Riemannian manifold. We refer to Theorem~\ref{liouville_critical_thm} for more details. 
 \section{Maximal interval of equivalence of Green functions}\label{sec_equivalence}\label{section_maximal_green}
 In the present section we provide a partial answer to Problem~\ref{pb_equivalence} concerning $G$-bounded perturbations under the quasimetric assumption.  
 This property of Green functions has been considered previously by several authors, for 
 example in \cite{FNV, KV, YP5}. 
\begin{defi}\label{def-qk}{\em 
A {\em quasimetric kernel} $K$ on a measure space $(M, \mu)$ is a measurable function from 
$M \times M \rightarrow (0, \infty]$ such that the following conditions hold.
\begin{enumerate}
\item The kernel $K$ is symmetric : $K(x, y) = K(y, x)$ for all  $ x, y \in M. $

\item The function $d := 1/K$ satisfies the quasi-triangle inequality
\begin{equation}\label{quasimetric}
d(x, y) \leq C(d(x, z) + d(z, y))\qquad \forall x, y, z \in M,
\end{equation}
 for some $C > 0$, called the {\em quasimetric constant} for $K$. 
\end{enumerate}
}
\end{defi}
\begin{remark}\label{rem-quasi}
	{\em
	Using Ptolemy inequality \cite[Lemma~2.2]{FNV}, it follows that if $G_{P}^{M}$ is a quasimetric kernel in the sense of Definition~\ref{def-qk}, then it satisfies the quasimetric inequality of \cite[Lemma~7.1]{YP5}. Therefore, in this case and in light of \cite[Lemma~7.1]{YP5}, if $W$ is $G$-semibounded perturbation, then $W$ is in fact, $G$-bounded perturbation.     
 }
\end{remark}

We are now in a position to state the main result of the present section.  We have

\begin{thm}\label{main-thm}
Let $P$ be a second-order, symmetric, subcritical elliptic operator of the form \eqref{symm_P} defined 
on noncompact Riemannian manifold $M$, and let $0\lneqq W\in L^q_\loc(M;\R)$, with  $q>N/2$ be a $G$-semibounded perturbation of $P$ in $M$.

 Assume further that $G_{P}^{M}$ is a quasimetric kernel. Then
 $$
 G^{M}_{P} \asymp G^{M}_{P - \varepsilon W}  \qquad \mbox{on}  \ M \times M
 $$
  for all $\varepsilon < \lambda_{0}=\lambda_{0}(P,W,M).$ Moreover,  
  $$E_+=S_+ \setminus \{ \lambda_0 \}  .$$ 
\end{thm}

Before proving Theorem~\ref{main-thm}, we recall some general results concerning the equivalence of Green functions.
We start with the following lemma.
 \begin{lem}[\cite{MM1,YP3,YP1}]\label{interval_equivalence}
Let $P$ be a second-order, subcritical elliptic operator of the form \eqref{operator} defined on noncompact Riemannian 
manifold $M$, and let $V\in L^q_\loc(M;\R)$  with  $q>N/2$ be a  $G$-bounded perturbation (that is, the $3G$-inequality \eqref{supremum} holds true).

 Then $P - \varepsilon V$ is subcritical and
\begin{equation}
G^{M}_{P} \asymp G^{M}_{P - \varepsilon V} \qquad \mbox{on}  \ M \times M
\end{equation}
for all $|\varepsilon| < (2C_0)^{-1}$. In particular,  $\lambda_{0}: = \lambda_{0}(P,V, M) > 0$.
\end{lem}

\begin{proof}
Consider the iterated Green kernel
\begin{equation}\label{ik2}
G^{(i)}_{P}(x, y) := \left\{
                       \begin{array}{ll}
                         G^{M}_{P}(x, y) & i=0, \\[4mm]
                        \int_{M} G(x, z) V(z) G^{(i-1)}_{P}(z, y) \dm(z) & i\geq 1.
                       \end{array}
                     \right.
\end{equation}
Then it follows from the hypothesis and an induction argument that
$$
|G^{(i)}_{P}(x, y)| \leq (C_0)^{i} G^{M}_{P}(x, y),
$$
where $C_0$ is given by \eqref{supremum}. Hence,
\begin{equation*}\label{first_estimate}
\sum_{i = 0}^{\infty} |\varepsilon|^{i} \left|G^{(i)}_{P}(x, y) \right| \leq \frac{1}{1 - C_0 |\varepsilon|} G^{M}_{P}(x, y),
\end{equation*}
provided $|\vge|<C_0^{-1}$. Fix  $|\vge|<C_0^{-1}$. Using a standard elliptic argument, it follow that the Neumann series
$$
H^{\varepsilon}_{P}(x, y) := \sum_{i = 0}^{\infty}  {\varepsilon}^{i} G^{(i)}_{P}(x, y)
$$
converges locally uniformly in $M$ to a Green function of $(P- \varepsilon V)u = 0.$  Moreover, for  $|\vge|<C_0^{-1}$, the positive minimal Green function $G^{M}_{P - \varepsilon |V|}$ exists, and by 
the minimality of the Green function it satisfies 
\begin{equation*}\label{upper_bound3}
0 \leq G^{M}_{P - |\varepsilon| |V|} (x, y) \leq \frac{1}{1 - |\varepsilon|C_0 } G^{M}_{P}(x, y).
\end{equation*}
Hence, $G^{M}_{P - \varepsilon V}$ exists, and by the generalized maximum  principle  we obtain
\begin{equation}\label{upper_bound}
0\leq G^{M}_{P - \varepsilon V} (x, y) \leq G^{M}_{P -| \varepsilon| |V|} (x, y) \leq \frac{1}{1 - |\varepsilon|C_0 } G^{M}_{P}(x, y).
\end{equation}
Using resolvent equation \cite[Lemma~2.4]{YP1}
$$G^{M}_{P - \varepsilon V} (x, y)= G^{M}_{P} (x, y)+ \varepsilon\int_{M} G_{P - \varepsilon V}(x, z)  V(z) G^{M}_{P}(z, y) \dm(z),$$
we obtain 
$$G^{M}_{P} (x, y)\leq G^{M}_{P - \varepsilon V} (x, y) + \frac{|\varepsilon| C_0}{1 - |\varepsilon|C_0} G^{M}_{P} (x, y). $$
Hence, for $|\varepsilon| < (2C_0)^{-1}$ we have
$$
\frac{1 - 2|\varepsilon|C_0 }{1 - |\varepsilon|C_0 } G^{M}_{P}(x, y)\leq G^{M}_{P - \varepsilon V} (x, y).
$$
Hence, the lemma follows.
\end{proof}
We recall a lemma regarding the convergence of the Neumann series of the iterated Green functions in the case of a perturbation by a potential $W$ with a definite sign.
\begin{lem}[Lemma~3.1, \cite{YP5}] \label{conv}
Let $P$ be a second-order, subcritical elliptic operator of the form \eqref{operator} defined on noncompact Riemannian 
manifold $M$, and let $W\in L^q_\loc(M;\R)$, with  $q>N/2$ be a nonzero, nonnegative potential 
such that $\lambda_{0}: = \lambda_{0}(P,V, M) > 0$.
Then
\item
\begin{equation}\label{c1}
\int_{M} G_{P}^{M} (x, z) W(z) G_{P}^{M}(z, y)\dm(z) < \infty,
\end{equation}
and for every $0 < \varepsilon < \lambda_{0}$, the Neumann series $\sum_{i = 0}^{\infty} \varepsilon^{i} G_{P}^{(i)}(x, y)$ converges to 
 $G_{P - \varepsilon W}^{M}(x, y)$
in the compact-open topology.
\end{lem}
\begin{proof}[Proof of Theorem~\ref{main-thm}]
	In light of Remark~\ref{rem-quasi} we may assume that $W$ is a $G$-bounded perturbation. 
	
Clearly, $E_+$ is an open set.	Indeed, if $\gl\in E_+$, then $W$ is $G$-bounded perturbation of $P-\gl W$, and by Lemma~\ref{interval_equivalence}, there exists $\vge_0>0$ such that 
   $(\gl-\vge_0,\gl+\vge_0)\subset E_+$  (see also \cite[Corollary~3.6]{YP2}). In particular,  $\gl_0\not \in E_+$.  
\medskip

Next, We claim that  $G^{M}_{P} \asymp G^{M}_{P - \varepsilon W}$ for all $\varepsilon <C_0^{-1}$.	

It follows from Lemma~\ref{interval_equivalence} that $G^{M}_{P} \asymp G^{M}_{P - \varepsilon W}$ for all $|\varepsilon| < (2C_0)^{-1}$. 
Moreover, by the generalized maximum principle, 
if $\varepsilon_1 < \varepsilon_2, $ then

\begin{equation}\label{g_maximum}
G^{M}_{P- \varepsilon_1 W} \leq G^{M}_{P- \varepsilon_2 W}.
\end{equation}
Therefore, $G^{M}_{P} \leq G^{M}_{P - \varepsilon W}$ for all $ 0\leq\varepsilon < \gl_0$. On the other hand, for  $0<\varepsilon<\frac{1}{C_{0}}$, 
we have  by \eqref{upper_bound} that \begin{equation}\label{eq177}
   G^{M}_{P}\leq  G^{M}_{P - \varepsilon W}\leq \frac{1}{1-\vge C_0}G^{M}_{P}.
\end{equation}

Fix $\varepsilon > 0$, and let 
$$G_0: =G^{M}_{P + \varepsilon W},\qquad  G_1: =G^{M}_{P - \frac{W}{2C_0}}, \qquad \ga:=\frac{\vge}{\vge+1/(2C_0)}\,.$$ 
In light of \cite[Theorem~{3.4}]{YP2} 
and \eqref{eq177}, we obtain
$$
 G_0=G^{M}_{P+\vge W}\leq G^{M}_{P} \leq (G_{1})^{\ga}(G_{0})^{1-\ga} \leq 2^{\ga} (G^{M}_{P})^{\ga}G_{0}^{1-\ga}.
$$
Therefore, 
$$G_{P+ \vge W}\leq G^{M}_{P}\leq 2^{2C_0\vge}G_{P+ \vge W}.$$
Hence, $G^{M}_{P - \varepsilon W} \asymp G^{M}_{P}$ for all $\varepsilon < \frac{1}{C_{0}}$.
 
 \medskip
 
Let $E_0: =\sup E_+$. Thus,   $0<\frac{1}{C_{0}} \leq E_0\leq \lambda_{0}$. We  claim that
$E_0 = \lambda_{0}$. 
Suppose to the contrary, that there exists $\delta > 0$ such that $E_0 + \delta < \lambda_{0},$ i.e.,
 $\frac{E_0 + \delta }{\lambda_0}  < 1.$ 

\medskip

Set $\dW := W(x) \!\dm(x)$, and define the iterated kernel 
\begin{equation*}
 K^{(i)}(x, y) := \left\{
                       \begin{array}{ll}
                      \left( E_0  + \delta \right) G^{M}_{P}(x, y) & i=0, \\[4mm]
                        \int_{M} G^{M}_{P}(x, z)  K^{(i-1)}(z, y) \dW(z) & i\geq 1,
                       \end{array}
                     \right.
\end{equation*}
and an operator $T :  L^{2}(M, \dW) \rightarrow L^{2}(M, \dW)$ by  
$$Tf(x): =  \left( E_0  + \delta \right)
 \int_{M}  G^{M}_{P}(x, y) f(y)
\dW(y).$$

 We claim that $T$ is well defined and $||T||_{L^2(M,\, \dW)} < 1.$

\medskip

Let $u$ be a positive supersolution of $(P - \lambda_{0} W) u = 0.$ Then it follows from \cite{YP2} that
$$
  \left( E_0  + \delta \right)  \int_{M} G^{M}_{P}(x, y) u(y) \dW(y) \leq \frac{\left( E_0  + \delta \right)u(x)}{\lambda_0}\, ,
$$
and
$$
\left(E_0 + \delta \right) \int_{M} u(x) G^{M}_{P}(x, y)\dW(x) \leq \frac{\left( E_0  + \delta \right) u(y)}{\lambda_0}\, .
$$
Therefore,   by Schur's test we obtain
 $$||T||_{L^2(M, \, \dW)} \leq \frac{ E_0  + \delta }{\lambda_0}  < 1.$$ 
 Define
\begin{equation}\label{converge_green}
H(x, y) := \sum_{i = 0}^{\infty} \left( E_0 + \delta \right)^{i} K^{(i)}(x, y) = \left( E_0 + \delta \right) G^{M}_{P - (E_0 + \delta) W}(x, y),
\end{equation}
which is well defined by Lemma~\ref{conv}.

 Hence, $T$ is a bounded linear integral operator on $L^{2}(M, \!\dW)$, with a quasimetric kernel $K$ and with a norm strictly less than $1$. Consequently,   \cite[Theorem~1.1]{FNV} implies that 
 \begin{equation}\label{verbi}
 \mathrm{e}^{ \frac{C_1K^{(1)}(x, y)}{K^{(0)}(x, y)}}K^{(0)}(x, y) \leq H(x, y) \leq  \mathrm{e}^{ \frac{C_2K^{(1)}(x, y)}{K^{(0)}(x, y)}} K^{(0)}(x, y),
 \end{equation}
 for some positive 
 constants $C_1$ and $C_2$.  
 
 Therefore, \eqref{verbi} and \eqref{converge_green} immediately imply
 \begin{equation}\label{eq-final}
  \left( E_0 + \delta \right) G^{M}_{P - (E_0 + \delta)W}(x, y) \leq K^{(0)}(x, y)\,  \mathrm{e}^{ \frac{C_2K^{(1)}(x, y)}{K^{(0)}(x, y)}}.
 \end{equation}
Now, observe that
$$
 \frac{K^{(1)}(x, y)}{K^{(0)}(x, y)} =    \frac{1}{G_{P}^{M}(x, y)} \int_{M} G_{P}^{M}(x, z) W(z)  G_{P}^{M}(z, y) \, \dm(z) \leq C_0.
$$
Hence, \eqref{eq-final} yields
$$
G^{M}_{P} (x, y)\leq  G^{M}_{P - (E_0 + \delta)W}(x, y) \leq  C G^{M}_{P} (x, y),
$$
where $C$ is a positive constant. This contradicts the maximality of $E_0$. 
Hence, $E_0 = \lambda_0.$
\end{proof}

\begin{rem}
{\rm
The validity of the conjecture $E_+=S_+ \setminus \{ \lambda_0 \}$, for a general nonnegative $G$-bounded perturbation $W$ of operator $P$ of the form \eqref{operator} remains open (cf.  \cite[Conjecture~3.7]{YP2}  and the counterexample \cite[Example~8.6]{YP5}).
}
\end{rem}

\section{Optimal Hardy-weights and $h$-bigness}\label{sec_hbig}

In the present section we study the $h$-bigness of  \emph{optimal Hardy-weights} $W\geq 0$ given by Theorem~\ref{thm_DFP}. Recall that $G$-bounded perturbations are 
non-$h$-big \cite{MM1}.
We note that under the conditions of Theorem~\ref{thm_DFP}, the operator $P_\gl : = P - \lambda W$ is subcritical
 in $M$ for all $\lambda < 1$. We have 
\begin{thm}\label{thm-hbig}
Consider the operator $P_\gl  := P - \lambda W$, and assume that 
\begin{itemize}
\item The operator $P$ is subcritical, and let  $G_\gf$ be a Green potential with respect to $P$, with a compactly supported smooth density $\gf$.  

\item There exists a positive solution $u$ of the equation $Pv=0$ in $M$ satisfying \eqref{DFP_cond}. 

\item $W$ is the corresponding optimal Hardy-weight given by \eqref{W-Hardy}.

\item  $0<\lambda <1$.
 
 \end{itemize}
 Set $\alpha_\pm := \frac{1 \pm \sqrt{1 - \lambda}}{2}$.
 
 Then $\gl W$ is $h_\pm$-big perturbations for the positive $P_\gl$-supersolutions
  $$h_\pm : = u^{(1 - \alpha_\pm)} (G_\gf)^{\alpha_\pm}.$$
\end{thm}
\begin{proof}
Let $K:=\supp \gf$. Since $\lambda = 4 \alpha_\pm (1 - \alpha_\pm)$, it follows that $h_\pm$ are indeed positive $P_\gl $-supersolutions in $M$, which are positive solutions of the equation 
$P_\gl v=0$ in $M\setminus K$ (see \cite[Theorem~3.1]{YP2}). 

Let $v_\pm$ be nonnegative solutions of $Pw=(P_\gl + \lambda W)w = 0 $ in $M$ satisfying 
$0 \leq v_\pm \leq h_\pm$. Suppose that $v_\pm>0$. So, 
$$
\frac{v_\pm(x)}{u(x)} \leq \left( \frac{G_\gf(x)}{u(x)} \right)^{\alpha_\pm}. 
$$
By our assumption, $\lim_{x \rightarrow \infty} \frac{G(x)}{u(x)} = 0$, therefore, $\lim_{x \rightarrow \infty} \frac{G_\gf(x)}{u(x)} = 0$. Consequently,
$$
\lim_{x \rightarrow \infty} \frac{v_\pm(x)}{u(x)} = 0.
$$
In light of  \cite[Proposition~6.1]{DFP}, we conclude $v_\pm$ are positive solutions of the equation $Pw=0$ in $M$ of minimal growth in a neighborhood of infinity in $M$. Hence 
$v_\pm$ are ground states, and $P$ is critical in $M$, a contradiction. Hence, we conclude $v_\pm \equiv 0.$
\end{proof}

\begin{rem}{\em 
		1. Since near infinity in $M$ we have
		$$ \left( \frac{G_\gf(x)}{u(x)} \right)^{\alpha_+}\leq  \left( \frac{G_\gf(x)}{u(x)} \right)^{\alpha_-},$$ 
		  it is enough to prove that   $\gl W$ is  $h_-$-big perturbation.

\medskip 
		  
		  2. Fix $x_0\in M$. We may consider the punctured manifold $M^*:=M\setminus \{x_0\}$, and let
$u$ is a positive solution of the equation $P  w=0$ in $M$, and $G(x):=G_P^M(x,x_0)$  satisfying \eqref{DFP_cond}. 
Let 
$$
W(x) := \frac{1}{4} \left| \nabla \log\left( \frac{G(x)}{u(x)} \right) \right|^2_{A(x)} \qquad \mbox{in } M\setminus  \{x_0\}.
$$ 
As in the proof of Theorem~\ref{thm-hbig}, it follows that for $0<\gl<1$, the potential $\gl W$ is $h_-$-big perturbations for $h_-:=u^{(1 - \alpha_-)} (G)^{\alpha_-}$. 
}
\end{rem}

\section{Critical Hardy-weights}\label{sec-critical-Hardy}
Throughout the present section we assume that $P$ is a subcritical operator in $M$ of the form \eqref{operator}.
We fix a positive  Radon measure $\gm$ on $M$ with a \textquoteleft nice\textquoteright\, nonnegative density $\gm(x)$. We denote $\mathrm{d}\gm=\gm(x)\dm$, and we assume that the corresponding {\em Green potential} $G_\gm$ is finite. That is, we assume that for some $x\in M$ (and therefore, for any $x\in M$) 
\begin{equation}\label{eq_finite_pot}
G_\gm(x):= \int_{M}\!\!\!  \Green{M}{P}{x}{y}\mathrm{d}\gm(y)<\infty.
\end{equation}
A sufficient condition for \eqref{eq_finite_pot} to hold is obviously, the existence of $k\geq 1$, and a positive (super)solution $\vgf^\star$ of the equation $P^\star u=0$ in $M^\star_k$ such that $\vgf^\star\in L^1(M^\star_k,\mathrm{d}\gm)$.

Set $$W_\gm(x):=\frac{\gm(x)}{G_\gm(x)}\,.$$

Since $PG_\gm=\gm$, it follows that the Green potential $G_\gm$ is a positive solution of the equation $(P-W_\gm)u=0$ in $M$, so, $\gl_0:=\gl_0(P,W_\gm,M)\geq 1$.
Moreover, since 
\begin{equation}\label{inv}
\int_{M}\Green{M}{P}{x}{y}W_\gm(y)G_\gm(y)\dm(y)
=G_\gm(x) \quad\forall x\in M,
\end{equation}
it follows that $G_\gm$ is a positive {\em invariant solution} of the equation  $(P-W_\gm)u=0$ in $M$ (see  \cite{YP2,YP17} and references therein). 

Without loss of generality, we assume that $0\in M$, and we denote $G(x):=\Green{M}{P}{x}{0}$. Since $PG=0$ in $M\setminus \{0\}$, and $G$ has minimal growth at infinity in $M$, it follows that for a given Green potential $G_\gm$ and for $\vge>0$ small enough, there exists a positive constant $C$ such that 
$$G(x)\leq CG_\gm(x) \qquad \forall x\in M\setminus B(0,\vge).$$
On the other hand,   let $V_\gm(x):=\frac{\gm(x)}{G(x)}$ in $M$. 
The following lemma characterizes Green potentials that are comparable (near infinity in $M$) to $G$ (see \cite[Corollary~4.7]{YP5}).
\begin{lemma}\label{lem_1}
	There exists a positive constant $C>0$ such that      
	\begin{equation}\label{eq_min_gr}
	C^{-1}G_\gm(x) \leq G(x) \qquad \forall x\in M
	\end{equation}
	if and only if $V_\gm$ is a $G$-semibounded perturbation of $P^\star$ in $M$.

\medskip
	
	Moreover, in this case, we  have $V_\gm\asymp W_\gm$ near infinity in $M$,  and in particular, $W_\gm$ is a $G$-semibounded perturbation of $P^\star$ in $M$.

	\medskip
	
	In addition, the convex set of all positive solutions $v$ of the equation $P^\star u=0$
	 in $M$ satisfying $v(0)=1$ is a bounded set in $L^1(M, \mathrm{d}\gm)$.    
\end{lemma}
\begin{proof}
 Assume first that $V_\gm$ is a $G$-semibounded perturbation of $P^\star$ in $M$. Then
	\begin{multline*}
	G_\gm(x)= \int_{M}\Green{M}{P}{x}{y}\frac{\gm(y)}{G(y)}G(y)\!\dm(y)=\\ \int_{M}\Green{M}{P}{x}{y}V_\gm(y)G(y)\dm(y)\leq CG(x) \qquad \forall x\in M,
	\end{multline*}
	and \eqref{eq_min_gr} holds.
	
	\medskip
	
	On the other hand, suppose that \eqref{eq_min_gr} holds. Consequently,
	\begin{equation}\label{eq-8}
	\int_{M}\!\Green{M}{P}{x}{y}V_\gm(y)G(y)\dm(y) =
	G_\gm(x)\leq C G(x)\quad \forall x\in M. 
	\end{equation}
Therefore,, $V_\gm$ is a $G$-semibounded perturbation of $P^\star$ in $M$. In particular, in this case we have $G_\gm\asymp G$ near infinity. This in turn, obviously implies that $V_\gm\asymp W_\gm$ near infinity.
	
	\medskip

	In addition, by \eqref{eq-8} we have
	$$
	\int_{M}\frac{\Green{M}{P}{x}{y}}{\Green{M}{P}{x}{0}}\dmu(y)=\int_{M}\frac{\Green{M}{P}{x}{y}V_\gm(y)G(y)}{G(x)}\dm(y) \leq C\qquad \forall x\in M.
	$$ 
	Therefore, the last assertion of the lemma follows from Fatou's lemma and the Martin representation theorem.  
\end{proof}
 The following lemma gives, in particular, a positive answer to  Problem~\ref{pb_equivalence} for the class of nonnegative $G$-semibounded perturbations of the form $W_\gm$.
\begin{lemma}\label{lem_2}
	Suppose that  \eqref{eq_min_gr} holds true, then $P-W_\gm$ is positive-critical in $M$ with respect to $W_\gm$, and $G_\gm$ is its ground state.
 Moreover, 
 $$SE_+(P,W_\gm,M)=S_+(P,W_\gm,M) =(-\infty,\gl_0(P,W_\gm,M))=(-\infty,1).$$
\end{lemma}
\begin{proof}
Recall that $G_\gm$ is a positive solution of the equation $(P-W_\gm)u=0$ in $M$. On the other hand, by our assumption $G_\gm\asymp G$ 
near infinity in $M$. Note that any positive supersolution $v$ of the equation $(P-W_\gm)u=0$ near infinity in $M$ is a positive supersolution  of the equation $Pu=0$ in this neighborhood, while $G$ is a positive solution of  $Pu=0$ of minimal growth near infinity. 

Consequently, 
$$G_\gm\leq CG\leq C_1 v \qquad \mbox{near infinity in } M.$$ Therefore, $G_\gm$ is a ground state of  the equation $(P-W_\gm)u=0$ in $M$, and $P-W_\gm$ is critical in $M$. 
Consequently, for any $0<\ga <1$ and $\varepsilon > 0$ sufficiently small, we have 
$$G \asymp \Green{M}{P-\ga W_\gm}{\cdot}{0}\asymp  G_\gm 
\qquad \mbox{in} \ M \setminus B(0, \varepsilon).
$$
Furthermore, in light of \cite[Corollary~3.6]{YP2}, $G \asymp \Green{M}{P-\ga W_\gm}{\cdot}{0}$ also for any $\ga<0$.  
 So, $SE_+(P,W_\gm,M)=S_+(P,W_\gm,M)=(-\infty,1)$. 

Moreover, since $P-W_\gm$ is critical in $M$, we have that  $P^\star-W_\gm$ is also critical in $M$. Denote by $u_\gm^\star$ its ground state. 
In particular,  $u_\gm^\star$ is a positive invariant solution of the corresponding equation \cite[Theorem~2.1]{YP2}. Therefore,
\begin{equation*}\label{gseq7}
\int_{M} \!\!G_\gm(x)W_\gm(x)u_\gm^\star(x)\!\dm(x)\!
\asymp\! \int_{M}\!\! G(x)W_\gm(x)u_\gm^\star(x)\!\dm(x)
\!=\!u_\gm^\star(0)\!<\!\infty .
\end{equation*} 
Hence, $P-W_\gm$ is positive-critical in $M$ with respect to $W_\gm$. 
\end{proof}

\begin{lemma}\label{lem-semismall}
 For $k\geq 2$, let $\chi_k$ be a smooth function on $M$ such that   
$$0\leq \chi_k(x)\leq 1, \mbox{ in } M \qquad \chi_k\!\!\upharpoonright_{M_{k-1}}=0,   \qquad \chi_k\!\!\upharpoonright_{M_{k}^\star}=1,$$
where $\{M_k\}$ is an exhaustion of $M$ (see Section~\ref{sec-pre}). Denote by $\gm_k(x):=\chi_k(x)\gm(x)$.
Assume further that
\begin{equation}\label{eq_uniform}
\lim_{k\to \infty}  \left\|\frac{G_{\gm_k}}{G}\right\|_{\infty; M_{k}^\star}=0. 
\end{equation}
Then $W_\gm$ is a semismall perturbation of the operator $P^\star$ in $M$, and for any $1\leq p\leq \infty$  the
integral operator
$$\mathcal{G}_\gm  f(x):=  \int_{M}  \Green{M}{P}{x}{y}W_\gm(y)f(y)\dm(y)$$
is compact on $L^{p}(\gf_p)$, where 
\begin{equation}\label{eq:2.9a}
\gf_p:=G_\gm^{-1}(G_\gm W_\gm u_\gm^\star)^{1/p}.
\end{equation}

Suppose in addition that $P$ is a symmetric  operator on
$L^2(M,W_\gm(x)\dm)$ with a core $C_0^\infty(M)$, 
Let
$\{(\vgf_k,\gl_k)\}_{k=0}^\infty$ be the set of the corresponding
pairs of eigenfunctions and eigenvalues  (counting multiplicity), where $\vgf_0:=G_\gm$ and $\gl_0=1$. Then  for every $k\geq 1$
there exists a positive constant $C_k$ such that 
\be\label{efest}
|\vgf_k(x)|\leq C_k \vgf_0(x) \qquad \mbox{in } M. 
\ee 
Furthermore, the function
$\vgf_k/\vgf_0$ has a continuous extension $\psi_k$ up to the Martin
boundary $\partial_{P}^MM$ of $P$ in $M$. 
\end{lemma}
\begin{proof}
The generalized maximum principle, and \eqref{eq_uniform} imply
\begin{equation}\label{eq_uniform1}
\lim_{k\to \infty}  \left\|\frac{G_{\gm_k}}{G}\right\|_{\infty; M}=0.
\end{equation} 
Hence,
\begin{multline*}
 \int_{M_{k}^\star}\Green{M}{P}{x}{y}W_\gm(y)G(y)\dm(y)=\int_{M_{k}^\star}\Green{M}{P}{x}{y}\frac{\gm(y)}{G_\gm(y)}G(y)\dm(y) \\
\leq  C \int_{M_{k}^\star}\Green{M}{P}{x}{y}\frac{\gm(y)}{G_\gm(y)}G_\gm(y)\dm(y)=CG_{\gm_k}(x)<\ge G(x)
 \qquad \forall x\in M,
\end{multline*}	
Consequently,  $W_\gm$ is a semismall perturbation of the operator $P^\star$ in $M$.
Therefore,  Theorem~\ref{thmcomp} implies that for any $1\leq p\leq \infty$  the
integral operator
$\mathcal{G}_\gm  f(x)$
is compact on $L^{p}(\gf_p)$, and its spectrum is $p$-independent and contained in the closed unit disk.  More precisely, the spectrum contains $0$, and besides, consists of at most a sequence of eigenvalues of
finite multiplicity which has no point of accumulation except $0$. 
Moreover, $\vgf_0=G_\gm$ is the unique nonnegative  eigenfunction of the operator
$\mathcal{G}_\gm \!\!\upharpoonright_{L^p(\gf_p)}$. Furthermore,  the corresponding eigenvalue $\gl_0=1$   is simple. 

\medskip

The statement concerning the symmetric case follows from Theorem~\ref{thmcomp}. We note that  by \cite{YP17}, the continuous extension $\psi_k$ of $\vgf_k/\vgf_0$ satisfies for $k\geq 1$
\bea\label{psineq} \psi_k(\xi)&=&
(\psi_0(\xi))^{-1}\gl_k\int_{M}K^M_{P}(z,\xi)W_\gm (z)\vgf_k(z)\dm(z)=
\nonumber\\ &
&\frac{\gl_k\int_{M}K^M_{P}(z,\xi)W_\gm(z)\vgf_k(z)\dm(z)}
{\int_{M}K^M_{P}(z,\xi)W_\gm(z)\vgf_0(z)\dm(z)}\qquad  \forall\xi\in\partial_{P}^MM,
 \eea 
where $K^M_{P}(\cdot,\xi)$ is the Martin kernel of $P$ in $M$ with a pole at $\xi\in\partial_{P}^MM$, and $\psi_0$ is the
corresponding continuous extension of $G_\gm/G$.
\end{proof}
\begin{remark}\label{rem-landscape}{\em 
		If $\gm=1$ and \eqref{eq_finite_pot} is satisfied, then $G_1$ is called the {\em torsion function} (see for example, \cite{vdBI13} and references therein). In a recent paper \cite{ADFJM}, D.~N.~Arnold, G.~David, M.~Filoche, D.~Jerison and S.~Mayboroda, considered the Green potential $W_1$ (which they called the {\em effective potential}) associated with a Schr\"odinger operator $L$  in a bounded Lipschitz domain $M\subset \R^N$. They showed  a remarkable connection between the Neumann eigenfunctions of $L$ and the  torsion function $G_1$ (which they call the {\em landscape function}) by proving that $W_1$ acts as an effective potential that
		governs the exponential decay of these eigenfunctions  and
		delivers information on the distribution of eigenvalues near the bottom
		of the spectrum.
	}
\end{remark}

\section{Finite torsional rigidity}\label{sec-torsion}
Throughout the present section we assume that $P$ is subcritical, symmetric operator on $L^2(M,\!\dm)$ of the form \eqref{symm_P}. Without loss of generality, we assume that $0\in M$, and we denote $G(x):=\Green{M}{P}{x}{0}$. In addition, we assume that $G_1\in L^1(M,\!\dm)$. So, we assume that the Green potential $G_1$   satisfies
$$
G_1(x):=\int_{M}  \Green{M}{P}{x}{y}\dm(y)<\infty, \quad \mbox{and }\; T(M):= \int_{M}  G_1(x)\dm(x)<\infty.$$
$G_1$ (resp., $ T(M)$) is called  the {\em torsion function} (resp., {\em torsional rigidity}) with respect to the operator  $P$ and the measure $\dm$. 
Note that if $G_1\asymp G$, then the finiteness of the torsion function $G_1$ is clearly equivalent to the finiteness of torsional rigidity $T(M)$.

\medskip

Following \cite{vdBI13}, we have

\begin{lemma}\label{lem-ftr}
Let $P$ be symmetric subcritical operator in $M$ with finite torsional rigidity. Assume further that there exists a function $$c:(0,\infty)\to (0,\infty)$$ 
such that
$k_P^M(x,y,t)$, the positive minimal heat kernel of $P$ in $(M,\!\dm)$, satisfies 
\begin{equation}\label{eq-hk}
k_P^M(x,y,t)\leq c(t)\qquad \forall t>0, x,y\in M.
\end{equation}
Then the spectrum of $P$ on $L^2(M,\!\dm)$ is discrete.

\medskip 

Suppose further that there exists $\gb\geq 0$  and $\tilde c>0$ such that 
$$c(t)\leq \tilde c \min\{t^{-N/2}, t^{-\gb/2}\} \qquad \forall t>0.$$
Then there exists a positive function  $C:\R_+\to\R_+$ such that
\begin{equation}\label{eq-glj}
 \gl_j\geq \min\left\{C(\beta) T(M)^{-2/(\gb+2)}j^{2/(\gb+2)},  C(N)  T(M)^{-2/(N+2)}j^{2/(N+2)}\right\},
 \end{equation}
where  $\{\gl_j\}_{j=0}^\infty$ is the increasing sequence of the eigenvalues  of $P$ (counting multiplicity). 
\end{lemma}

\begin{proof}
Since $$ G_1(x)= \int_M \int_0^\infty k_P^M(x,y,t) {\rm d}t\, \dm,$$
by Tonelli's theorem, it follows that  for any $0<\ga<1$, we have 
$$T(M)= (1-\ga) \int_0^\infty {\rm d}t  \int_{M\times M} k_P^M(x,y,(1-\ga)t)\dm(y)\dm(x).$$
In light of \eqref{eq-hk} and the semigroup property, we have
\begin{align}\label{trace_class}
T(M) & \!\geq \!(1-\ga)\!\! \int_0^\infty \!\!\!\!\big( c(\ga t)\big)^{\!-1} \! {\rm d}t  \!\!  \int_{M\times M}\!\!\!\! \!\!
k_P^M(x,y,(1-\ga)t)k_P^M(x,y,\ga t)\!\dm(y)\!\dm(x) \notag \\[2mm]
 & =(1-\ga) \int_0^\infty \big( c(\ga t)\big)^{-1} \, {\rm d}t \int_{M} k_P^M(x,x,t)\dm(x).
\end{align}
It follows that the heat operator $k_P^M$ is trace class. So, for each  $t>0$ we have
$$
\int_{M} k_P^M(x,x,t)\dm(x)=\sum_{j=0}^\infty \exp(-\lambda_j t)<\infty,
$$
where $\{\gl_j\}$ is the nonincreasing sequence of all the eigenvalues of $P$ (counting multiplicity).  In particular, $P$ has a discrete $L^2(M,\dm)$-spectrum.   

Estimate \eqref{eq-glj} is obtained as in \cite[Theorem~2]{vdBI13}. Indeed, by \eqref{trace_class} we have 
$$T(M) \!\geq\!  (1-\ga) (\tilde{c})^{-1} \!\! \!\int_0^\infty\!\!\! (\ga t)^{\gb/2}\! \sum_{j=0}^\infty\! {\mathrm e}^{-\lambda_j t}\!\dt\!\geq\!  (1-\ga)(\tilde{c})^{-1}  j  \!\! \int_0^\infty\! \!\!(\ga t)^{\gb/2}  {\mathrm e}^{-\lambda_j t}\!\dt.
$$
Recall that $$\int_0^\infty t^{\gg} {\mathrm e}^{- \ell t}\dt= \frac{\Gg(\gg+1)}{\ell^{\gg+1}}\,.$$ 
 Hence, for  $\ga := \frac{\beta}{\beta+2}$, we obtain \eqref{eq-glj} with $C(\beta)$ given by 
$$
C(\beta) :=   \frac{\beta^{\frac{\beta}{\beta + 2}}}{\beta + 2} \left( \frac{2\Gg((\gb+2)/2)}{\tilde{c}} \right)^{2/(\gb+2)}\!\!. \qquad \qquad\qquad  \qedhere
$$
\end{proof}
\section{Liouville comparison principle}\label{Sec_4.1}
The present section is devoted to the study of Liouville comparison principle for {\em nonsymmetric} elliptic operators. The following theorem should be compared with  Theorem~\ref{YP07-thm} and \cite[Theorem~2.3]{ABG}.
\begin{theorem}\label{liouville_critical_thm}
 Let $M$ be a smooth, noncompact, connected manifold of dimension $N$. Consider two operators  
\begin{equation*}
P_k  :=  \mathcal{L}_{k} - V_k  \qquad    k = 1,2,
\end{equation*}
where each $\mathcal L_k$ is of the form \eqref{operator},  and $V_k \in L^{p}_{\loc}(M ; \mathbb{R})$, where $p>N/2$. Let $\overline{V} (x) = \max \{ V_1(x), V_2(x) \}.$ Suppose 
that there exists $K_1 \Subset K\Subset M$ such that $\mathcal{L}_{1} = \mathcal{L}_{2}$ in $M \setminus K_1$, and $P_k\geq 0$ in $M \setminus K_1$, for $k=1,2$. 

\medskip

Let $G_k$ be a positive supersolution of the equation $P_k u=0$ in $M \setminus K_1$,  such that $G_k$ is a positive solution of the equation $P_k u=0$ in $M \setminus K$ of minimal growth at infinity in $M$, where $k=1, 2$. Suppose that 
\begin{equation}\label{eq-V1-V2}
\frac{|V_1 - V_2|}{2} \leq W:= \frac{1}{4} \left| \nabla \log\left( \frac{G_1}{G_2} \right) \right|^2_{A} \qquad \mbox{in } M\setminus K.
\end{equation}
Then

(a)  $\mathcal{L}_{1}-\overline V \geq 0$ in $M\setminus K$.

\medskip

(b) Assume further the that the following assumptions hold true: 

\begin{enumerate}

\item The operator $P_1$ is critical in $M$, and let $\Phi \in \mathcal{C}_{P_1}(M)$ be its ground state. 
\medskip

\item $P_2 \geq 0$ in $M$, and there exists a real function $\Psi \in W^{1,2}_\loc(M)$ such that $\Psi_+ \neq 0$ and $P_2 \Psi \leq 0$ in $M$.

\medskip
\item The following inequality holds: 
$$
\Psi_+ \leq C \Phi \qquad \mbox{in }  M.
$$
\end{enumerate}
Then the operator $P_2$ is critical in $M$ and $\Psi$ is its ground state. In particular, the equation $P_2 v=0$ admits a unique positive supersolution in $M$. 
Moreover, $\Psi\asymp \Phi$ in $M$. 
\end{theorem}
\begin{proof} 
The proof relies on criticality theory, the supersolution construction \cite{DFP}, and on the well known ``maximal $\varepsilon$-trick''. We denote the restriction of the operators 
 $ \mathcal{L}_k$ on $M \setminus K_1$ by $ \mathcal{L}$.

(a)  We note that $U:=(G_1 G_2)^{1/2}$ is a positive solution of the equation 
 \begin{equation}\label{optimal_Hardy_liouville}
\left(\mathcal{L} - \left(\frac{V_1 + V_2}{2} \right) - W \right) v = 0 \quad \ \mbox{ in } 
M \setminus K ,
\end{equation}
where $W$ is given in \eqref{eq-V1-V2}. Since   $$
\overline{V} = \max \{ V_1(x), V_2(x) \} = \frac{V_1 + V_2}{2} + \frac{|V_1 - V_2|}{2},
$$
assumption \eqref{eq-V1-V2} implies that $U$ is a positive supersolution of the equation $(\mathcal{L}-\overline V)u \geq 0$ in $M\setminus K_1$. Hence,  
$\mathcal{L}-\overline V \geq 0$  in $M\setminus K_1$. 

\medskip

(b) Let $\overline G$ be a positive solution of the equation $(\mathcal{L}-\overline V)u=0$ in $M \setminus K$ of minimal growth at infinity in $M$.
Then by the generalized maximum principle and the fact that $G_1$ has minimal growth at infinity in $M$ we have that 
\begin{equation}\label{G1leqU}
G_1\leq C_1\overline G \leq C_2 U = C_2(G_1 G_2)^{1/2} \qquad  \mbox{in } M \setminus K.
\end{equation}
 Hence, $G_1 \leq C_3 G_2$ in $M \setminus K$.

Since $\Gf \leq \tilde C G_1$ in  $M \setminus K$, and $G_2$ has minimal growth at infinity in $M$ for $P_2$, we have that for any positive supersolution $f$ of the equation $P_2 u=0$ in $M$ we have
\begin{equation}\label{5.7}
\Psi_+ \leq C \Phi \leq C\tilde C G_1\leq C\tilde C C_3G_2 \leq C_4 f \qquad  \mbox{in } 
 M \setminus K. 
\end{equation}  
 
 Define
   $$\varepsilon_0 = \mbox{max} \{ \varepsilon : \varepsilon \Psi(x) \leq f(x) \quad \forall x \in M \}.$$
  In light of \eqref{5.7},  it follows that $\varepsilon_0>0$ is well defined, and hence,
$w(x) := f(x) - \varepsilon_0 \Psi (x) $ is a nonnegative supersolution of the equation  
$P_2 v =0 $ in $M$. 
 
By the strong maximum principle, either  $w  > 0$  or $w= 0$ in $M$.
Let us assume that $ w > 0$. Then by replacing $f$ with $w$ and repeating
  the above argument, we conclude that there exists $\delta > 0$ such that $f - (\vge_0+\gd) \Psi > 0$, 
  which contradicts the maximality of $\varepsilon_0$. Hence, $w = 0$ in $M$, which in turns implies that
   $$
   \Psi (x) = \Psi_+ =\vge_0 f(x)>0 \qquad \forall x \in M .
   $$
   Since $f$ is an arbitrary positive supersolution of $P_2 u = 0$ in $M$, it follows that $P_2$ is critical in $M$ and $\Psi$ is its ground state. The assertion $\Psi\asymp \Phi$ in $M$ follows now from \eqref{5.7} since $\Psi (x) = \Psi_+>0$ in $M$ and $G_2$ is a positive solution of the equation $P_2 u=0$ in $M \setminus K$ of minimal growth at infinity in $M$.
    \end{proof}

 \begin{rem}
 {\rm
 Under the assumptions of Theorem~\ref{liouville_critical_thm}, it follows that the positive minimal Green functions of $P_k$   in $M\setminus K$, where $k=1,2$,  are semiequivalent. Moreover, \eqref{G1leqU} implies that these Green functions are also semiequivalent to the positive minimal Green function of $\mathcal L -\overline V$ in $M\setminus K$. We note that using  
 \cite[Theorem~4.3]{YP95} it follows that under the assumptions of Theorem~\ref{liouville_critical_thm}, the operators $\mathcal L_k -\overline V$ might be supercritical in $M$. 
 }
 \end{rem}
 The following example demonstrates that inequality \eqref{eq-V1-V2} might not 
 hold and still the Liouville comparison principle holds true.
 \begin{example}\label{ex_a}
	{\rm
 Let $P_1 = -\Delta$, $V_1=0$ in $\mathbb{R}^2$. Then it is well known that $P_1$ is critical 
 and $1$ is the corresponding ground state.  
 Let $P_2 = - \Delta - V_2$ be nonnegative  
 in $\mathbb{R}^2$, where  $V_2\in L^{\infty}(\mathbb{R}^2)$ is a radially symmetric potential that satisfies 
\begin{equation}\label{V}
V_2(x) = \frac{\lambda}{|x|^2} \quad \mbox{in} \ \mathbb{R}^2 \setminus B(0,1),
\end{equation}
 where $\lambda < 0$ be any real number. A straightforward computation yields 
 $G_2(x) := |x|^{- \sqrt{-\lambda}} $ is positive solution in 
 $\mathbb{R}^2 \setminus B(0,1)$ of minimal growth at infinity in $\R^2$
 for $P_2$. Also $G_1(x) = 1$ is a positive solution of minimal growth at infinity in $\R^2$ for $P_1$, so, $G_1\not \asymp G_2$ near infinity.  Note that
 $$
 \frac{|V_1 - V_2|}{2} = \frac{|\lambda|}{2|x|^2} > \frac{|\lambda|}{4|x|^2}
  = \frac{1}{4} \left| \nabla \log\left( \frac{G_1}{G_2} \right) \right|^2.
 $$
 On the other hand, the Liouville comparison principle (Theorem~\ref{YP07-thm}) applies for the above $P_1$ and $P_2$, since these operators are symmetric.  In particular, if the equation $P_2 u=0$ in $M$ admits a nonzero, nonnegative, bounded subsolution, then $P_2$ is critical in $M$.   
  }
\end{example}
Next, we slightly modify the above example by adding a drift term to the Laplacian.   
  \begin{example}\label{ex_b}
  	{\rm 
  Consider the operator		
  $$P_1 = -\Delta - b \, \dfrac{\chi_{B(0, 1)^{*}} }{r} \partial_r  \quad  \mbox{ in } \
 \mathbb{R}^2, 
 $$
 and $V_1=0$,  where $r : =|x|$, $b$ is a negative constant, and $\chi_{B(0, 1)^{*}}$ is the
 indicator function of $B(0, 1)^{*}:=\mathbb{R}^2 \setminus B(0, 1)$. 
 Then $P_1$ is critical in   $\mathbb{R}^2$, with a ground state equals $1$. Let 
 $$P_2 : = - \Delta - b \, \dfrac{\chi_{B(0, 1)^{*}} }{r} \partial_r - V_2,$$
  where $V_2\in L^{\infty}(\mathbb{R}^2)$ satisfies \eqref{V}, such that 
   $P_2\geq 0$  in $\R^2$.  Then as before we easily find that 
 $G_2(x) := |x|^{\frac{-b - \sqrt{ b^2 - 4 \lambda }}{2}}$ is a positive solution in $B(0, 1)^{*}$ of minimal growth at infinity in $\R^2$ for $P_2$. 
 Also, $G_1(x) = 1$ is a positive solution of minimal growth at infinity in $\R^2$ for $P_1$, so, $G_1\not \asymp G_2$ near infinity.  We note that for $|x|>1$ we have
 $$
 \frac{1}{4} \left| \nabla \log\left( \frac{G_1}{G_2} \right) \right|^2 = \frac{|\lambda|}{4 |x|^2} - 
 \frac{b^2}{8 |x|^2} \left[ \sqrt{1 + \frac{4|\lambda|}{b^2}} - 1\right]
 $$
This immediately yields  as before 
 $$
 \frac{|V_1 - V_2|}{2}= \frac{|\lambda|}{2|x|^2} > \frac{1}{4} \left| \nabla \log\left( \frac{G_1}{G_2} \right) \right|^2.
  $$ 
  
  On the other hand, Theorem 2.14 applies for the above $P_1$ and $P_2$, since the operator $P_1$ is symmetric in 
  $L^2( \mathbb{R}^2, {\rm d}m),$ where 

$$ \dm=  m(x)\dx : =\left\{\begin{array}{ll}
\dx & \text{if $x \in B(0, 1)$\,},
\\[2mm]
|x|^{b}\dx & \text{if  $x \in \mathbb{R}^2 \setminus B(0, 1)$\,}.
\end{array}
\right.
$$
In particular, if the equation $P_2 u=0$ in $M$ admits a nonzero, nonnegative, bounded subsolution, then $P_2$ is critical in $M$.    
 }
 \end{example}
\section{Green function estimate on the hyperbolic space} 
\label{sec_green_function_hyperbolic}
As an application of our results, we  study the behaviour of the positive minimal Green function of the 
shifted Laplacian on ${\mathbb H}^N$, the real hyperbolic space. It is well known that a Cartan-Hadamard manifold $M$ whose sectional curvatures is bounded above by a strictly negative constant satisfies the Poincar\'e inequality, or in other words, the bottom of the $L^2$-spectrum of the  Laplace-Beltrami  on $M$ is strictly positive. The most important example of such a manifold is ${\mathbb H}^N$. Let   
 $\Delta_{\mathbb{H}^N}$ denote the Laplace-Beltrami operator on the hyperbolic space, then 
the \emph{generalized principal eigenvalue} of $-\Delta_{\mathbb{H}^N}$ is given by 
$$
\lambda_0(-\Delta_{{\mathbb H}^{N}}, \textbf{1}, {\mathbb H}^N)= \frac{(N-1)^2}{4}.
$$
Moreover, by using  explicit bounds for the heat kernel on ${{\mathbb H}^N}$ (see e.g. \cite{DA}) one can show that the nonnegative
 operator 
 $$P:= -\Delta_{\mathbb H^N}-(N-1)^2/4$$
admits a positive minimal 
Green function (for $N\geq 2$). In other words, $P$ is subcritical in ${{\mathbb H}^N}$. 

\medskip

Fix $x_0\in\mathbb H^N$, and let $G(x) : = G^{\mathbb H^N}_{-\Delta_{\mathbb H^N}}(x, x_0)$. For $0< \gl < 1$, let  
$$0<\ga_-<1/2<\ga_+<1$$
 be the roots of the equation  $\lambda = 4 
\alpha(1 - \alpha)$. Using the supersolution construction \cite{DFP}, it follows that  $G^{\alpha_{\pm}}$ are  solutions of the 
equation 
$$
(-\Delta_{\mathbb H^N} - \lambda W) G^{\alpha_{\pm}} = 0 
\quad \mbox{in }   \mathbb{H}^N \setminus  \{ x_0 \}, \quad \mbox{where } W := \dfrac{1}{4} \dfrac{|\nabla G|^2}{|G|^2}\, .
$$

\medskip

  The asymptotic of $W$ is given by the following lemma.
\begin{lem}
Let $N \geq 2$ and $r : = {\rm d}(x, x_0).$ Then $W(r)$ satisfies
$$
W(r) = \frac{(N-1)^2}{4} + \frac{(N-1)^3}{N + 1} \mathrm{e}^{-2r} + o(\mathrm{e}^{-2r}) \quad \mbox{as} \  r \rightarrow \infty.
$$
\end{lem}

\begin{proof}
For the hyperbolic space ${\mathbb H}^N$, the Green function of the Laplace-Beltrami operator is given by 
$$ G(x)= \tilde G(r) : = \int_{r}^{\infty} (\sinh s)^{-(N-1)} \, {\rm d}s. $$
We have 
\begin{align*}
(\sinh s)^{-(N-1)} = 2^{N-1}  \mathrm{e}^{-(N-1)s} (1 -  \mathrm{e}^{-2s})^{-(N-1)}.
\end{align*}
Therefore,  $r \rightarrow \infty$ yields 
$$
(\sinh r)^{-(N-1)} = 2^{N-1} \left(  \mathrm{e}^{-(N-1)r} + (N-1) \mathrm{e}^{-(N+1)r} + o\big( \mathrm{e}^{-(N+1)r}\big) \right).
$$ 
Furthermore, as $r\rightarrow \infty$ we have 
$$
\int_{r}^{\infty} \!\!\!(\sinh s)^{-(N-1)}\!\ds \!= \!2^{N-1} \!\!\left[\!\frac{1}{N-1}  \mathrm{e}^{-(N-1)r} \!+\! \frac{N-1}{N+1}  \mathrm{e}^{-(N+ 1)r} \!\!+ o\big( \mathrm{e}^{-(N+1)r}\big)\!\right]\!.
$$
Hence, as $r \rightarrow \infty$ we have
$$
W(r) \!=\! \frac{1}{4} \left[ \frac{(\sinh r)^{-2(N-1)}}{\left(\int_{r}^{\infty} (\sinh s)^{-(N-1)} {\rm d}s\right)^2} \right]=\frac{(N-1)^2}{4}  +\frac{(N-1)^3}{N + 1}  \mathrm{e}^{-2r} + o( \mathrm{e}^{-2r}).
$$ 
\end{proof}
Now we state the following perturbative result. 
\begin{thm}
Let $N \geq 2$ and $0<\lambda < 1.$ Then there holds 
\begin{equation}\label{green_estimate_hyperbolic}
G^{\mathbb H^N}_{-\Delta_{\mathbb H^N} - \lambda \frac{(N-1)^2}{4}}(x,x_0) \asymp G^{\mathbb H^N}_{-\Delta_{\mathbb H^N} - \lambda W}(x,x_0) 
\asymp G^{\alpha_+}(x) \quad \mbox{in } \mathbb{H}^N \setminus B(x_0, 1),
\end{equation}
where $\lambda = 4 \alpha_+(1 - \alpha_+)$ and $ \frac{1}{2}<\alpha+ <1$.
\end{thm}

\begin{proof}
Recall that $G^{\mathbb H^N}_{-\Gd_{\mathbb H^N} -\gl W}(x,x_0)$ is a positive solution of minimal 
growth at infinity of the equation $(-\Gd_{\mathbb H^N} -\gl W)v=0$  in $\mathbb{H}^N$. On the other hand, $$\lim_{r\to\infty}\frac{G^{\alpha_+}(r)}{G^{\alpha_-}(r)}=0.$$
Therefore,  \cite[Proposition~6.1]{DFP} implies that $G^{\alpha_+}$ is also a positive solution of minimal 
growth at infinity of the equation $(-\Gd_{\mathbb H^N} -\gl W)v=0$  in $\mathbb{H}^N$. Thus, 
 $$G^{\mathbb H^N}_{-\Delta_{\mathbb H^N} - \lambda W}(x,x_0) \asymp G^{\alpha_+}(x) \qquad \mbox{in } \mathbb{H}^N \setminus B(x_0, 1).$$
 Hence, it remains to prove that
 $$G^{\mathbb H^N}_{-\Delta_{\mathbb H^N} - \lambda \frac{(N-1)^2}{4}}(x,x_0) \asymp G^{\mathbb H^N}_{-\Delta_{\mathbb H^N} - \lambda W}(x,x_0) 
  \qquad \mbox{in } \mathbb{H}^N \setminus B(x_0, 1).$$
Note that  for $r \rightarrow \infty,$ we have 
$$\gl W(r)- \gl \frac{(N-1)^2}{4}= \gl \frac{(N-1)^3}{N + 1}  \mathrm{e}^{-2r} + o( \mathrm{e}^{-2r}).$$ 
Consequently, Remark~\ref{rem_sp} implies that it suffices to show that $\tilde W(r) :=  \mathrm{e}^{-2r} + o( \mathrm{e}^{-2r})$ is a small perturbation of the operator $P_\lambda := -\Delta_{\mathbb H^N} -
 \lambda \frac{(N-1)^2}{4}$ in $\mathbb H^N$.

\medskip

We follow  the approach of Ancona \cite[corollary~6.1]{AN}. Let us choose $\Phi(r) :=  \mathrm{e}^{-(2 - \varepsilon)r}$ with $0<\varepsilon <1$. 
Then it follows 
\begin{equation}\label{limit_hyperb}
\lim_{r \rightarrow \infty} \dfrac{\Phi(r)}{\tilde W(r)} = + \infty.
\end{equation}
Moreover, $\Phi$ is nonnegative, nonincreasing  and $\int_{0}^{\infty} \Phi (r) {\rm d} r < \infty.$ Therefore, by \cite[Theorem~1]{AN}, 
we conclude 
\begin{equation}\label{green_equiv_hyperb}
G^{\mathbb H^N}_{P_\lambda} \asymp G^{\mathbb H^N}_{P_\lambda + \Phi(r) \textbf{1}_{\mathbb{H}^N \setminus B(x_0, R)}} \quad \mbox{in } \mathbb H^N\times \mathbb H^N 
\end{equation}
for large $R$. Consequently,  \eqref{green_equiv_hyperb} and arguments given in \cite{YP3, YP1} implies that $\Phi$ is a $G$-bounded perturbation of $P_\gl$ in $\mathbb{H}^N$. 

Hence, it follows from \eqref{limit_hyperb} that $\tilde W$ is a small perturbation for $P_\lambda$. In particular,  by Remark~\ref{rem_sp} we have
\begin{equation*}
G^{\mathbb H^N}_{P_\lambda} \asymp G^{\mathbb H^N}_{-\Gd_{\mathbb H^N} -\gl W} \quad \mbox{in } \mathbb H^N\times \mathbb H^N\setminus \{(x,x)\mid x\in \mathbb H^N\}.   
\end{equation*}
Thus, \eqref{green_estimate_hyperbolic} follows.
 \end{proof}

\medskip

 \begin{center}
 	{\bf Acknowledgments} 
 \end{center}
 

D.~G.  is supported in part by an INSPIRE faculty fellowship (IFA17-MA98) and is grateful to the Department of Mathematics at the Technion for the hospitality during his visit. He also  acknowledges the support of the Israel Council for Higher Education (grant No. 32710877). The authors acknowledge the support of the Israel Science Foundation (grant 970/15) founded by the Israel Academy of Sciences and Humanities.

\end{document}